\documentclass{article}
\usepackage{amsmath,amsthm,amssymb}
\usepackage{url,xspace}
\usepackage{comment}
\usepackage{enumerate}
\newtheorem{theorem}{Theorem}
\newtheorem{corollary}[theorem]{Corollary}
\newtheorem{proposition}[theorem]{Proposition}

\newtheorem{lemma}[theorem]{Lemma}

\theoremstyle{definition}
\newtheorem{definition}[theorem]{Definition}

\newtheorem{example}[theorem]{Example}
\newtheorem{question}[theorem]{Question}

\newcommand{\But}{\textrm{BH}}

\newcommand{\lcm}{\mathrm{LCM}}

\renewcommand{\Re}{\mathrm{Re}}

\newcommand{\abk}{\allowbreak}

\newcommand\nth{\textsuperscript{th}\xspace}

\title{\textbf{\Large{Morphisms of Butson classes}}}
\author{
\textsc{Ronan Egan}
				\thanks{\textit{E-mail: ronan.egan@math.uniri.hr}}\\
\textit{\footnotesize{Department of Mathematics}}\\
\textit{\footnotesize{University of Rijeka, Croatia}}\\
\textsc{Padraig \'O Cath\'ain}
				\thanks{\textit{E-mail: pocathain@wpi.edu}}\\
\textit{\footnotesize{Department of Mathematical Sciences}}\\
\textit{\footnotesize{Worcester Polytechnic Institute}}\\}

\begin{document}
\thispagestyle{empty}
\maketitle

\begin{abstract}
We introduce the concept of a morphism from the set of Butson Hadamard matrices over $k\nth$ roots of unity to the set of Butson matrices over $\ell\nth$ roots of unity.  As concrete examples of such morphisms, we describe tensor-product-like maps which reduce the order of the roots of unity appearing in a Butson matrix at the cost of increasing the dimension. Such maps can be constructed from Butson matrices with eigenvalues satisfying certain natural conditions.  Our work unifies and generalises Turyn's construction of real Hadamard matrices from Butson matrices over the $4\nth$ roots and the work of Compton, Craigen and de Launey on `unreal' Butson matrices over the $6\nth$ roots.  As a case study, we classify all morphisms from the set of $n \times n$ Butson matrices over $k\nth$ roots of unity to the set of $2n\times 2n$ Butson matrices over $\ell\nth$ roots of unity where $\ell < k$.

\end{abstract}

\section{Introduction}

Let $M$ be an $n \times n$ matrix with entries in the complex numbers $\mathbb{C}$. If every entry $m_{ij}$ of $M$ has modulus bounded by $1$,
then Hadamard's theorem states that $|\det(M)| \leq n^{n/2}$. Hadamard himself observed that a matrix $M$ meets this bound with equality if and
only if every entry in $M$ has modulus $1$, and every pair of distinct rows of $M$ are orthogonal (with respect to the usual Hermitian inner product) \cite{Hadamard1893}. While Hadamard's name has become associated with real matrices meeting the bound, his original paper is not restricted to this case.

It is well known that a real Hadamard matrix of order $n$ (i.e. a matrix with entries in $\{\pm 1 \}$ meeting the Hadamard bound) can exist only if $n = 1, 2$ or $n \equiv 0 \mod 4$.
The Hadamard conjecture states that these necessary conditions are also sufficient. Since the discovery in 2005 of a real Hadamard matrix of order 428, the smallest
open case is $n = 668$ \cite{Kharaghani428}. Asymptotic existence results are also available in the real case: building on pioneering work of Seberry, Craigen has shown that
for any odd integer $m$, there exists a Hadamard matrix of order $2^{\alpha \log_{2}(m) + 2}m$ where $\alpha < 1$ is a constant \cite{CraigenAsymptotic,SeberryAsymptotic}.
Despite the existence of many constructions for Hadamard matrices, the density of integers $n$ for which it is known that there exists a Hadamard matrix of order $n$ is essentially
given by the density of the Paley matrices (i.e. the density of the primes, all other constructions contribute a higher order correction term) \cite{deLGordonAsymptotic}.

There is an analogue of the Hadamard conjecture for any number field: the usual Hadamard conjecture concerns the existence of Hadamard matrices over $\mathbb{Q}$.
Examples of matrices are known in which the entries are described as roots of modulus $1$ of certain polynomial equations, see matrix $A_{8}^{(0)}$ of the database \cite{Karol} for example. We will be exclusively interested in matrices whose entries are roots of unity in this paper. We write $M^{\ast}$ for the Hermitian transpose of $M$.

\begin{definition}
A matrix $H$ of order $n$ with entries in $\langle \zeta_{k}\rangle$ is \textit{Butson Hadamard} if $HH^{\ast} = n I_{n}$.
We write $\But(n,k)$ for the set of such matrices.
\end{definition}

There exists a large literature on real Hadamard matrices, which have close ties to symmetric designs, certain binary codes and difference sets \cite{BJL,HoradamHadamard}.
The Butson classes $\But(n,4)$ and $\But(n,6)$ have received some attention in the literature, due to applications in signal processing and connections
to real Hadamard matrices \cite{CCdeL,Turyn}. In fact, if $\But(n,4)$ is non-empty, then $\But(2n,2)$ is necessarily non-empty. This motivates the
\textit{complex Hadamard conjecture}: that there exists a Hadamard matrix of order $n$ with entries in $\{\pm 1, \pm i\}$ whenever $n$ is even.

More generally, de Launey and Winterhof have described necessary conditions for the existence of Hadamard matrices whose entries are roots of unity \cite{deLauneyExistence,WinterhofExistence}. A typical application of these results shows that $\But(n, 3^{a})$ is empty whenever the square-free part
of $n$ is divisible by a prime $p \equiv 5 \mod 6$. In contrast to $\But(n,2)$ and $\But(n,4)$, there does not appear to be any consensus in the
literature on sufficient conditions for $\But(n,k)$ to be non-empty for general values of $k$. Apart from the obvious examples of character tables of
abelian groups, relatively few constructions for matrices in $\But(n,k)$ are known. An early result of Butson shows that for all primes $p$, the set $\But(n,p)$
is empty unless $n \equiv 0 \mod p$, and that $\But(2p,p)$ is non-empty \cite{Butson}. Further constructions and a survey of known results are given by Agaian \cite{Agaian}.

\section{Morphisms of Butson matrices}

We define a (complete) morphism of Butson matrices to be a function $\But(n,k) \rightarrow \But(r,\ell)$.
The Kronecker product is perhaps the best-known example of a morphism: if $M \in \But(m,d)$ and
$H$ is any matrix in $\But(n,k)$ then $H \otimes M \in \But(mn, \ell)$ where $\ell$ is the least common
multiple of $k$ and $d$. To be explicit, the map $- \otimes M: \But(n,k) \rightarrow \But(mn,\ell)$
is a Butson morphism.

We will restrict our attention to morphisms which come from embeddings of matrix algebras (as the tensor product does),
such morphisms can be considered generalised plug-in constructions. In this case, we say a morphism $\But(n,k) \rightarrow \But(mn,\ell)$ is of degree $m$.  We will be particularly interested in the construction
of morphisms where $k$ does not divide $\ell$. We will also relax our conditions to allow \textit{partial morphisms},
where the domain is a proper subset of $\But(n,k)$; typically we impose a restriction on matrix entries of $H$.
We introduce the concept of a \textit{sound pair} to collect necessary and sufficient conditions for our main existence theorem.

We define $\zeta_{k} = e^{2\pi i/k}$, and set $G_{k} = \langle \zeta_{k} \rangle$.  We define $H^{\phi}$ to be the entrywise application of $\phi$ to $H$ whenever $\phi$ is a function defined on the entries of $H$, and we write $H^{(r)}$ for the function which replaces each entry of $H$ by its $r\nth$ power.

\begin{definition} \label{defn:sound}
Let $X,Y \subseteq G_{k}$ be fixed.
Suppose that $H \in \But(n,k)$ such that every entry of $H$ is contained in $X$,
and that $M \in \But(m,\ell)$ such that every eigenvalue of $\sqrt{m}^{-1}M$ is contained in $Y$.

Then the pair $(H, M)$ is $(X,Y)$-\textit{sound} if
\begin{enumerate}
\item For each $\zeta_{k}^{i} \in X$, we have $\sqrt{m}^{1-i} M^{i} \in \But(m, \ell)$.
\item For each $\zeta_{k}^{j} \in Y$, we have $H^{(j)} \in \But(n,k)$.
\end{enumerate}
We will often say that $(H,M)$ is a sound pair if there exist sets $X$ and $Y$ for which $(H, M)$ is $(X, Y)$-sound.
\end{definition}

Often the first condition of a sound pair is satisfied only when $i$ is odd: this occurs
in particular when $\mathbb{Q}[\zeta_{k}]$ contains no elements of absolute value $\sqrt{m}$.
But there do exist matrices satisfying this condition where $m$ is not a perfect square:
one example is given by the square of the matrix
\begin{equation}
\label{24throots}
M_{24} = \left[ \begin{array}{rr} 1 & 1 \\ i & -i \end{array}\right] ,\,\,\,\, \sqrt{2}^{-1}M_{24}^{2} =  \left[ \begin{array}{rr} \zeta_{8} & \zeta_{8}^7 \\ \zeta_{8} & \zeta_{8}^{3} \end{array}\right]
\end{equation}
which belongs to $\But(2,8)$. This example also shows that one should distinguish carefully between $\ell$ and the
degree of the cyclotomic field generated by the entries of $M$ in Definition \ref{defn:sound}. A natural question arises
about the smallest field containing the entries of $\sqrt{m}^{1-i}M^{i}$ for all $i$. Since we are interested exclusively
in Butson matrices, this will be the largest cyclotomic field contained in $\mathbb{Q}[\zeta_{\ell}, \sqrt{m}]$, which is at
most a quadratic extension of $\mathbb{Q}[\zeta_{\ell}]$, so its torsion units will be at most $2\ell\nth$ roots of unity if
$\ell$ is even and $4\ell\nth$ roots of unity if $\ell$ is odd.

Before we prove our main theorem, we recall that for an $n\times n$ matrix $A$, and an
$m \times m$ matrix $B$, the Kronecker products $A \otimes B$ and $B \otimes A$ are similar matrices.
There exists a permutation matrix $P_{mn}$ such that $P_{mn} (A\otimes B) P_{mn}^{-1} = B \otimes A$.
We call this matrix the \textit{Kronecker shuffle}. The locations of the non-zero entries in $P_{mn}$ can be precisely described.

\begin{proposition}[Rose, \cite{Rose}] \label{prop:blockdiag}
For any $m,n \in \mathbb{N}$, the $nm \times nm$ \textit{Kronecker shuffle} matrix $P_{mn}$ is
\[ P_{mn} = \left[ \delta^{im - \lfloor(i-1)/n\rfloor(mn-1)}_{j + m - 1} \right]_{1\leq i, j \leq mn}\,. \]
If $M$ is an $mn \times mn$ matrix consisting of $n\times n$ diagonal blocks,
then $P_{mn}M P_{mn}^{-1}$ has $m\times m$ blocks down the diagonal and is zero elsewhere.
\end{proposition}

We now show that the conditions of a sound pair are sufficient to guarantee that a plug-in construction
gives a Hadamard matrix.

\begin{theorem}\label{thm:main}
Let $H \in \But(n,k)$ and $M \in \But(m,\ell)$ be Hadamard matrices.
Define a map $\phi: \zeta_{k}^{i} \mapsto \sqrt{m}^{1-i} M^{i}$, and write
$H^{\phi}$ for the entrywise application of $\phi$ to $H$.
If $(H, M)$ is a sound pair then $H^{\phi} \in \But(mn,\ell)$.
\end{theorem}

\begin{proof}
Fix sets $X,Y \subseteq G_{k}$ such that $(H, M)$ is $(X, Y)$-sound.
By the first property of a sound pair, for every $\zeta_{k}^{i} \in X$,
the image $\phi(\zeta_{k}^{i}) = \sqrt{m}^{1-i} M^{i} \in \But(m, \ell)$. Hence
every entry in $H^{\phi}$ is in $\langle \zeta_{\ell} \rangle$. To show that $H^{\phi}$ is a Hadamard matrix, it will suffice
to show that every eigenvalue of $H^{\phi}$ has absolute value $\sqrt{mn}$.

Every Hadamard matrix is diagonalisable. Write $A$ for a matrix
such that $AMA^{-1}$ is diagonal, and define
$\psi(\zeta_{k}^{i}) = \sqrt{m}^{1-i} AM^{i}A^{-1}$.
Observe that if $\zeta_{k}^{\alpha}$ is the $t\nth$
diagonal entry of $AMA^{-1}$, then
$\psi(\zeta_{k}^{i})_{t,t} = \zeta_{k}^{i\alpha}$, where by hypothesis $\zeta_{k}^{\alpha} \in Y$.
Writing $H^{\psi}$ for the entrywise application of $\psi$ to $H$, it follows easily that
\[H^{\psi} = (I_{n} \otimes A) H^{\phi} (I_{n} \otimes A^{-1})\,. \]
Now, $H^{\psi}$ is an $mn \times mn$ block matrix in which each $m\times m$
block is diagonal. Applying Proposition \ref{prop:blockdiag}, we obtain a block diagonal matrix,
\[ P_{mn} H^{\psi} P_{mn}^{-1} = \textrm{diag} [ B_{1}, B_{2}, \ldots, B_{k}] \,.\]

To conclude, observe that $B_{t}[i,j] = \sqrt{m} \psi(h_{i,j})_{t,t} = \sqrt{m}h_{i,j}^{\alpha}$ with the notation chosen above.
So $B_{t} = \sqrt{m}H^{(\alpha)}$, and by the second property of a sound pair, each eigenvalue of $B_{t}$ has absolute value $\sqrt{mn}$.
This argument holds for each block $B_{t}$, so since $P_{mn}H^{\psi}P_{mn}^{-1}$ and $H^{\phi}$ are similar, we have $H^{\phi}\in \But(mn,\ell)$ as required.
\end{proof}

In the special case that every eigenvalue of $M$ is a primitive $k\nth$ root of unity, the second condition of Definition \ref{defn:sound} is vacuous:
raising each entry to its $\alpha\nth$ power for $\alpha$ coprime to $k$ is a field automorphism. Such an operation preserves both the modulus of entries of $H$ and the orthogonality of rows.
These are examples of \textit{global equivalence operations} as considered by de Launey and Flannery \cite{deLauneyFlannery}.
The first condition of Definition \ref{defn:sound} still places restrictions on the entries of $H$: it can never contain $1$, for example;
since the image of $\zeta_{k}^{0} = 1$ in $H^{\phi}$ never has entries of modulus $1$. Next we identify some sound pairs of
special interest, since they correspond to complete morphisms.

\begin{corollary}\label{cor:main}
Suppose that $M \in \But(m,\ell)$,
that all eigenvalues of $M$ are primitive
$k\nth$ roots of unity,
and that $\sqrt{m}^{1-i} M^{i}$ is
Hadamard for all $i$ coprime to $k$.
Let $d = \prod_{i=1}^{r} p_{i}^{\alpha_{i}}$ such that $p_{i}^{\alpha_{i}+1} \mid k$
for $1 \leq i \leq r$. Then there exists a complete Butson morphism $\But(n,d) \rightarrow \But(mn,\ell)$.
\end{corollary}

\begin{proof}
Under the hypotheses, the primitive $k\nth$ roots of unity contain a translate of the $d\nth$ roots of unity.
If $H \in \But(n,d)$ then every entry of $\zeta_{k}H$ is a primitive $k\nth$ root. The conditions of
Definition \ref{defn:sound} are satisfied with $X$ and $Y$ both taken to be the
primitive $k\nth$ roots of unity. The claim follows from Theorem \ref{thm:main}.
\end{proof}

The following construction of Turyn illustrates Corollary \ref{cor:main}.

\begin{example}[Turyn, \cite{Turyn}]
\label{TurynEx}
Let $M_{8}$ be the matrix
\[ M_{8} = \left[ \begin{array}{rr} 1 & 1 \\ -1 & 1 \end{array}\right]\,.\]
It is easily verified that the eigenvalues of $M_{8}$ are $\sqrt{2} \zeta_{8}$
and $\sqrt{2}\zeta_{8}^{7}$. Likewise, it can be verified that
\[ \frac{1}{2} M_{8}^{3} =  \left[ \begin{array}{rr} -1 & 1 \\ -1 & -1 \end{array}\right] \,\]
that $2^{-2}M_{8}^{5} = -M_{8}$ and that $2^{-3}M_{8}^{7} = -M_{8}^{3}$. (In fact $M_{8}^{8} = 16I_{2}$.)
Thus $M_{8}$ satisfies all the conditions of Corollary \ref{cor:main} with $m = 2$ and $\ell = 8$.
So for any $H \in \But(n,4)$, the pair $(\zeta_{8} H, M_{8})$ is sound. We recover Turyn's famous
morphism: $\But(n,4) \rightarrow \But(2n,2)$.
\end{example}

Next, we illustrate the full generality of Theorem \ref{thm:main}.

\begin{example}[Compton, Craigen, de Launey \cite{CCdeL}]
\label{CCdeL ex}
Let $M_{6}$ be the following matrix (which is in fact similar to the matrix given by Compton, Craigen and de Launey):
\[ M_{6} = \left[ \begin{array}{rrrr} 1 & 1 & 1 & 1 \\ -1 & 1 & -1 & 1 \\ -1 & 1 & 1 & -1 \\ -1 & -1 & 1 & 1 \end{array}\right]\,. \]
One computes that the eigenvalues of $2^{-1}M_{6}$ are the primitive sixth roots of unity, each with multiplicity $2$.
Likewise, one can check that $2^{-1}M_{6}^{2}$ is Hadamard and that $M_{6}^{3} = -8I_{4}$. As a result, $2^{-3}M_{6}^{4}$ and $2^{-4}M_{6}^{5}$ are Hadamard.

In the definition of a sound pair, we can take
$X = \{\zeta_{6}, \zeta_{6}^{2}, \zeta_{6}^{4}, \zeta_{6}^{5}\}$ and
$Y = \{\zeta_{6}, \zeta_{6}^{5} \}$. Since raising each entry to its first power is the identity map on
$\langle \zeta_{6}\rangle$ and raising an entry to its $5\nth$ power is complex conjugation,
the restrictions placed on $H$ by $Y$ are vacuous. Since $M_{6}^{3}$ is a scalar matrix,
we cannot allow $-1$ as an entry in $H$. Compton, Craigen and de Launey call a matrix in $\But(n,6)$ with
entries in $X$ \textit{unreal}. We have constructed a partial morphism: $\But(n,6) \rightarrow \But(4n,2)$
with domain the unreal matrices.
\end{example}

We conclude this section with a question which we feel should have a positive answer.

\begin{question}\label{coprime Q}
If the eigenvalues of $M \in \But(m, \ell)$ are all primitive $k\nth$ roots of unity, is it true that $\sqrt{m}^{1-i}M^i \in \But(m, \ell)$ for all $i$ coprime to $k$?
\end{question}

\section{Construction of Butson morphisms} \label{sec:partial}

\subsection{Partial morphisms}

We begin by describing a relationship between certain sets of mutually unbiased bases and partial morphisms of Hadamard matrices.

\begin{definition}
Let $V$ be an $n$-dimensional vector space over $\mathbb{C}$ carrying the usual Hermitian inner product.
Two orthonormal bases $B_{1}$ and $B_{2}$ of $V$ are \textit{unbiased} if $|\langle u,v \rangle| = n^{-1/2}$
for all $u \in B_{1}$ and $v \in B_{2}$. A set of bases is \textit{mutually unbiased} if each pair is
unbiased.
\end{definition}

We refer to sets of mutually unbiased bases as MUBs. Such objects are of substantial interest in quantum physics,
and are well studied in the literature. Normalising so that $B_{0}$ is the standard normal basis, every other $B_{i}$
is necessarily represented by a Hadamard matrix. Being mutually unbiased means that $B_{i}B_{j}^{\ast}$ is again Hadamard
for all $i \neq j$. The largest possible number of MUBs in $\mathbb{C}^{d}$ is $d+1$. Maximal sets of MUBs are known to exist in
every prime power dimension, and in no other dimensions. It is believed that the maximal number of MUBs in dimension $6$ is three,
but this is an outstanding open problem \cite{MUBSsurvey}.

To construct sound pairs, we require a Hadamard matrix $M$ for which a specified set of its powers are Hadamard. Since $M^{i} (M^{j})^{\ast} = M^{i-j}$,
a subset $\mathcal{B}$ of the powers of $M$ is a set of MUBs if and only if $M^{i-j}$ is Hadamard for all $M^{i}, M^{j} \in \mathcal{B}$.
A remarkable construction of Gow yields Butson matrices with the property that every power is either Hadamard or scalar.
The construction uses the representation theory of extra-special $p$-groups of exponent $p$ over $\mathbb{C}$ for odd $p$, these are the so-called discrete Heisenberg groups. (For $p = 2$, Gow uses generalised quaternion groups.)

\begin{theorem}[Gow, \cite{GowEasy, GowHard}]\label{gow}
Let $q = 2^{a}$ such that $q+1$ is prime.
There exists $M_{q} \in \But(q, 4)$ and $M'_{q} \in \But(q^{2}, 2)$
such that $M_{q}$ has only primitive $(q+1)^{\textrm{st}}$ roots of unity
as eigenvalues.
\end{theorem}

\begin{proof}
Since $M'_{q} \in \But(q^{2}, 2)$, every eigenvalue of $q^{-1}M'_{q}$
is a complex number of modulus $1$. Since $q+1$ is prime, the eigenvalues
are even $(q+1)^{\textrm{st}}$ roots of unity (though not necessarily all primitive).
Gow proves that the trace of $q^{-1}M'_{q}$ is $-q$. Together with the irreducibility of the
cyclotomic polynomials over $\mathbb{Q}$, this shows that every eigenvalue of $q^{-1}M'_{q}$
is a primitive $(q+1)^{\textrm{st}}$ root of unity.

The proof is similar in the complex case: again, the matrix $\sqrt{q}^{-1}M_{q}$ is of multiplicative order $q+1$.
So the eigenvalues are $(q+1)^{\textrm{st}}$ roots of unity, and the minimal polynomial divides the cyclotomic polynomial
$\Phi_{q+1}(x)$. Factoring $q+1$ over the Gaussian integers and applying a generalisation of Eisenstein's criterion,
shows that $\Phi_{q+1}(x)$ remains irreducible over $\mathbb{Q}[i]$, so every eigenvalue of $\sqrt{q}^{-1}M_{q}$ is again a
$(q+1)^{\textrm{st}}$ root of unity.
\end{proof}

The smallest example of Gow's theorem is in dimension $2$, where Gow gives the matrix
\[ 2^{-1/2}\zeta_{8}M = \frac{1+i}{2}  \left[ \begin{array}{rr} -1 & i\\ 1 & i \end{array}\right] \,.\]
Some care is necessary in interpreting the matrices of Theorem \ref{gow} as partial morphisms, however.
While Gow's unitary matrix is defined over $\mathbb{Q}[i]$ and has order $3$, the corresponding Hadamard matrix
(denoted $M$ above) has order $24$. In fact, if there exists $M \in \But(2, \ell)$ such that $2^{-3/2}M^{3} = I_{2}$
then $8 \mid \ell$. Interestingly, applying the Turyn morphism to $\zeta^{5}_{8}M$ yields a matrix in $\But(4, 2)$ which is
similar to the Compton-Craigen-de Launey matrix of Example \ref{CCdeL ex}. For a Fermat prime $p > 3$, it is well known that $p-1$
is a perfect square, and we do not need to pass to a larger field.

\begin{corollary}\label{Fermat}
Suppose that $p = 2^a + 1$ is a Fermat prime, $p> 3$. Then there exists $M_{p} \in \But(2^a, 4)$ and $M'_{p} \in \But(2^{2a}, 2)$ with only primitive $p\nth$ roots as eigenvalues.
Hence, for any $H \in \But(n,p)$ with no entries equal to $1$, the pairs $(H, M_{p})$ and $(H, M'_{p})$ are sound.
\end{corollary}

Gow has also constructed large sets of MUBs as powers of a single matrix in odd prime power dimensions. But when $q+1$ is not prime, the conditions imposed by Definition
 \ref{defn:sound} on $H$ may be impossible to satisfy. Gow's $5 \times 5$ matrix has eigenvalues $\{-1, \pm \omega, \pm \omega^{2}\}$,
 which require that $H^{(2)}$ and $H^{(3)}$ both be Hadamard. But this implies that $H$ is generalised Hadamard over
 the cyclic group of order $6$: generalised Hadamard matrices have been constructed only over $p$-groups, and it has been conjectured
that no such matrices exist over groups of composite order. (See Section 2.10 of \cite{deLauneyFlannery} for the definition of generalised Hadamard matrices, and
\cite{Brock} for some non-existence results.)

\begin{example}
It seems rather difficult to realise Gow's work as an explicit construction for partial morphisms.
We constructed several examples computationally. The follows matrix has as eigenvalues the primitive fifth roots of unity.
\[ M_{5} = \left[ \begin{array}{rrrr} -1 & -1 & -1 & -1 \\ 1 & -1 & 1 & -1 \\ i & i & -i & -i \\ i & -i & -i & i \end{array}\right]\,.\]
Equivalently, $M_{5}$ induces a partial morphism $\But(n, 5)\rightarrow \But(4n, 4)$ where the domain consists of matrices have no entries
equal to $1$. Allowing negations, we obtain a map $\But(n, 10) \rightarrow \But(4n, 4)$ with domain the unreal matrices (i.e. those containing no real entries).
\end{example}

These morphisms generalise the Compton-Craigen-de Launey result which is the case $a = 1$ of Corollary \ref{Fermat}.

\subsection{New morphisms from old}\label{secnew}

Suppose that there is a complete morphism $\But(n,d) \rightarrow \But(nm,\ell)$ where $M$, $d$ and $k$ are defined as in Corollary \ref{cor:main}.
The next result allows us to construct new morphisms over larger roots of unity.

\begin{theorem}\label{thm:newfromold}
If $(\zeta_{k}H,M)$ is $(X,Y)$-sound for all $H \in \But(n,d)$ where $Y$ consists only of primitive $k\nth$ roots of unity, then $(\zeta_{kt}H, \zeta_{t}M)$ is sound for any $t$ coprime to $k$.
\end{theorem}

\begin{proof}
The eigenvalues of $\zeta_{t}M$ are primitive $kt\nth$ roots of unity.
Let $T \subset \{1,\ldots,k-1\}$ be the set of all $i$ such that $\sqrt{m}^{1-i}M^i \in \But(m, \ell)$.
By hypothesis, $M$ induces a complete morphism $\But(n,d) \rightarrow \But(nm,\ell)$;
so $T$ contains an arithmetic progression of length $d$, say $D \subseteq T$.

Furthermore, $\sqrt{m}^{1-kt}(\zeta_{t}M)^{kt} = I_m$, and $\sqrt{m}^{1-a}(\zeta_{t}M)^a \in \But(m, \ell t)$ whenever $a \equiv i \mod k$ and $i \in T$.
Let $T'$ be the set of all $a$ such that $\sqrt{m}^{1-a}(\zeta_{t}M)^a \in \But(m, \ell t)$. Define $D' \subseteq T'$ to be the set exponents such that $a \mod t$ is in $D$.
Then $D'$ is an arithmetic progression of length $dt$.  Thus $(\zeta_{kt}H, \zeta_{t}M)$ is $(X',Y')$-sound for all $H \in \But(n,d)$, where $X' = \{\zeta_{kt}^{i} \; : \; i \in T'\}$ and $Y'$ is the set of primitive $kt\nth$ roots of unity.
\end{proof}

\begin{corollary}
If there exists a complete morphism $\But(n,d) \rightarrow \But(nm,\ell)$, there exists a complete morphism $\But(n,dt) \rightarrow \But(nm,\lcm(t,\ell))$.
\end{corollary}

As an illustration of Theorem \ref{thm:newfromold} we generalise Turyn's morphism.

\begin{example}
Let $M_{8}$ be the matrix of Example \ref{TurynEx}. The eigenvalues of $\zeta_{3}M_{8}$ are primitive $24\nth$ roots of unity, and $\sqrt{2}^{1-i}(\zeta_{3}M_{8})^i$
is Hadamard for all odd $i$. So for any $H \in \But(n,12)$, the pair $(\zeta_{24}H,\zeta_{3}M_{8})$ is sound, and thus there is a complete morphism $\But(n,12) \rightarrow \But(2n,6)$.
\end{example}

We can also generalize to partial morphisms.  For example, the matrix
\[ M_{12} = 2^{-1}\left[ \begin{array}{rrrr} 1 & 1 & 1 & 1 \\ 1 & -1 & -1 & 1 \\ 1 & 1 & -1 & -1 \\ -1 & 1 & -1 & 1 \end{array}\right] \]
is of order 12, has primitive $12\nth$ roots as its eigenvalues, and $2^{-i}M_{12}^{i}$ is Hadamard for all $i \in T = \{1,2,4,\abk 5,7,8,10,11\}$.
Then $2^{1-i}(\zeta_{5}M_{12})^i$ is Hadamard for all $i \in T' = \{1,\ldots,60-1\} \setminus \{3,6,\ldots,57\}$, and the eigenvalues of $\zeta_{5}2^{-1}M_{12}$ are primitive $60\nth$ roots of unity.  So $T'$ contains two arithmetic progressions of length $20$.  Thus, with the appropriate choice of $X$ and $Y$, $(\zeta_{60}H,\zeta_{5}M_{12})$ is sound for any $H \in \But(n,20)$, and $(K,\zeta_{5}M_{12})$ is sound for any $K \in \But(n,60)$ such that no entry in $K$ is a $20\nth$ roots of unity.  So there is a complete morphism $\But(n,20) \rightarrow \But(4n,10)$ and a partial morphism $\But(n,60) \rightarrow \But(4n,10)$.

\subsection{Construction of complete morphisms}

The construction of large sets of MUBs is a challenging open problem, even without the added restriction that all of the bases arise as powers of a single matrix.
As a result, we do not expect in general to find subgroups of $\langle m^{-1/2} M \rangle$ which contain many Hadamard matrices. On the other hand, there seem to
be few restrictions on cosets containing many Hadamard powers. This is precisely the intuition behind Corollary \ref{cor:main}.

It seems natural (though not strictly necessary) to study matrices $M$ such that $M \in \But(m, \ell)$ and the unitary matrix $m^{-1/2}M$ has as eigenvalues only primitive $k\nth$ roots of unity. The characteristic polynomial of $m^{-1/2}M$ is a divisor of a power of the cyclotomic polynomial $\Phi_{k}(x)$. So the possible characteristic polynomials then depend on the factorisation of $\Phi_{k}(x)$ in $\mathbb{Q}[\zeta_{\ell}, \sqrt{m}]$. The following Proposition gives an easy construction of larger complete morphisms from small ones.

\begin{proposition}\label{prop:completes}
Suppose that $4 \mid k$ and that the eigenvalues of $M \in \But(m, \ell)$ are primitive $k\nth$ roots of unity. If $H \in \But(n, \ell)$ is Hermitian, then the eigenvalues of $H \otimes M \in \But(mn, \ell)$ are all primitive $k\nth$ roots of unity.
\end{proposition}

\begin{proof}
The eigenvalues of $H$ are real, and so all lie in $\{ \pm \sqrt{n}\}$. The eigenvalues of a tensor product are the products of the eigenvalues
of the component matrices, and by hypothesis, $\zeta_{k}^{i}$ is primitive if and only if $\zeta_{k}^{i + n/2}$ is.
\end{proof}

The first interesting examples of complete morphisms occur when $M$ is real and its eigenvalues are primitive $8\nth$ roots of unity. The matrix $M_{8}$ in Example \ref{TurynEx} is one such. It is somewhat unusual in that the size of the matrix is smaller than the degree of $\Phi_{8}(x)$; in fact its characteristic polynomial comes from an exceptional factorisation of cyclotomic polynomials related to Sophie Germain's identity. Larger examples of complete morphisms can be constructed from the Turyn example via tensor products. We combine Proposition \ref{prop:completes} and Example \ref{TurynEx}: since real symmetric Hadamard matrices are easily constructed at orders $1, 2$ and $4t$ for all $t \leq 10$, we will focus on the cases $n \equiv 4 \mod 8$. The classification of $\But(10, 4)$ and $\But(14, 4)$ has been completed up to Hadamard equivalence by Lampio, Sz\"{o}ll\H{o}si and \"Osterg\r{a}rd: Hermitian matrices exist at both orders \cite{LSO}. We have not yet managed to construct a real Hadamard matrix of order $36$
with eigenvalues in the set $\{6\zeta_{8}, 6\zeta_{8}^{3}, 6\zeta_{8}^{5}, 6\zeta_{8}^{7} \}$.

\begin{theorem}
There exists a complete morphism $\But(n, 4) \rightarrow \But(2mn,2)$ whenever there exists a Hermitian matrix in $\But(m, 4)$.
In particular, there exist such morphisms for $m = 1$ and for all even $m \leq 8$.
\end{theorem}

It seems more challenging to construct complete morphisms $\But(n,k) \rightarrow \But(mn, \ell)$ for which $k$ and $\ell$ are coprime.

\begin{question}
What is the smallest $m$ for which there is a complete morphism $\But(n,3) \rightarrow \But(mn,2)$?
\end{question}

If $M$ yields a complete morphism as in the question, then the eigenvalues of $M$ are primitive $9k^{\textrm{th}}$ roots of unity for some $k \in \mathbb{N}$. If $k = 4$, the characteristic polynomial of $(12t)^{-1/2}M$ is necessarily $(x^{12} - x^{6} + 1)^{t}$, where $M$ is a matrix of order $12t$. The characteristic polynomial of $M$ can be obtained from that of the unitary matrix via the substitution $x^{k} \mapsto (12t)^{(n-k)/2}x^{k}$.

\subsection{Eigenvalues of $\But(2,\ell)$}

In this section we classify all $M \in \But(2,\ell)$ for which the corresponding unitary matrix has finite multiplicative order, and study morphisms derived
from these matrices. We begin with a lemma on roots of unity which will constrain the eigenvalues of $M$.

\begin{proposition} \label{prop:roots}
Suppose that $\alpha$ and $\lambda$ are roots of unity such that $\Re(\alpha) = \sqrt{2} \Re(\lambda)$. Then up to negation and complex conjugation, $[\alpha, \lambda]$ is one of $[i, i]$, $[1,\zeta_{8}]$, $[\zeta_{8}, \zeta_{6}]$.
\end{proposition}

\begin{proof}
Recall that $\Re(\zeta) = 1/2(\zeta + \zeta^{-1})$. So we require the solutions of the identity
\[ \frac{\alpha + \alpha^{-1}}{\lambda + \lambda^{-1}} = \sqrt{2}\,. \]
Expanding as a quadratic in $\alpha$ and solving, we have that
\[ \sqrt{2}\alpha = -(\lambda + \lambda^{-1}) \pm \sqrt{\lambda^{2} + \lambda^{-2}} \,.\]
Applying field automorphisms, we may assume that $\lambda = \zeta_{k}$ for some suitable integer $k$.
Suppose that $k \geq 8$, then $\lambda + \lambda^{-1} = 2\Re(\lambda) \geq \sqrt{2}$, while $|\sqrt{\lambda^{2} + \lambda^{-2}}| \leq \sqrt{2}$.
Since $k \geq 8$, we have that $\Re(\lambda^{2}) > 0$, hence the right hand side of the equation is real and negative. The unique solution of
absolute value $\sqrt{2}$ occurs when $k = 8$. The solutions with $k < 8$ can be found by inspection.
\end{proof}

\begin{corollary}\label{cor:eigratio}
Suppose that $M \in \But(2, \ell)$ such that $2^{-1/2}M$ has finite multiplicative order.
Then $\lambda_{1}\lambda_{2}^{-1} \in \{-1, \pm i, \pm \zeta_{3}\}$.
\end{corollary}

\begin{proof}
For an arbitrary $2 \times 2$ Butson matrix
\[ M = \left[ \begin{array}{rr} \alpha & \beta \\ \gamma & \delta \end{array}\right] \in \But(2, \ell), \]
we observe that $(\alpha^{-1/2} \delta^{-1/2}) M$ has real trace, and hence its eigenvalues are
conjugate. Furthermore, the ratio of the eigenvalues of $H$ is preserved by scalar multiplication.
Orthogonality of rows forces $\beta = -\gamma^{\ast}$. In fact, different choices of $\beta$ yield similar matrices.
So to compute the ratio of the eigenvalues of $M$, it suffices to consider matrices of the form
\begin{equation}
\label{normform}
M_{\alpha} = \left[ \begin{array}{rr} \alpha & \alpha \\ -\alpha^{*} & \alpha^{*} \end{array}\right]\,.
\end{equation}
We compute the eigenvalues explicitly:
\[ \lambda = \Re(\alpha) + i \sqrt{ 2 - \Re(\alpha)^{2}}, \,\,\, \lambda^{\ast} = \Re(\alpha) - i \sqrt{ 2 - \Re(\alpha)^{2}}\,.\]
Since clearly $\Re(\alpha) \leq 1$, the second term is always purely imaginary and we have $\Re(\alpha) = \Re(\lambda)$.
The matrix $2^{-1/2}M_{\alpha}$ has finite order if and only if its eigenvalues are roots of unity. Setting $2^{-1/2}\lambda = \lambda'$, we require the classification of  the pairs of roots of unity $(\alpha, \lambda')$ for which $\Re(\alpha) = \sqrt{2}\Re(\lambda')$. Now apply Proposition \ref{prop:roots}.
\end{proof}

Corollary \ref{cor:eigratio} reduces the analysis of the $2 \times 2$ Butson matrices to three cases. We deal with the traceless matrices
separately, since the proof is short, and we never obtain morphisms for which $k \geq \ell$.

\begin{lemma}\label{traceless}
Suppose that $M \in \But(2, \ell)$ such that $2^{-1/2}M$ has finite multiplicative order with $k\nth$ root of unity eigenvalues.
Then if $\lambda_{1} = -\lambda_{2}$, $k \leq \ell$.
\end{lemma}

\begin{proof}
Let $a = m_{11}$, and so $m_{22} = -a$. Orthogonality of the rows of $M$ implies that $m_{12} = ab$ and
$m_{21} = ab^{\ast}$ for arbitrary $b$ of modulus $1$. It is easily verified that different choices of $b$ produce similar
matrices and that the eigenvalues of $M$ are $\pm \sqrt{2} a$. Hence $k \leq \ell$.
\end{proof}

Now we turn our attention to the remaining cases of Corollary \ref{cor:eigratio}. We will restrict attention to matrices $M$ in $\But(2, \ell)$ for
which the eigenvalues of $2^{-1/2}M$ are primitive $k\nth$ roots. From this information, we easily obtain a classification of all complete morphisms of order $2$.

\begin{theorem}\label{order2class}
Suppose that $M \in \But(2, \ell)$, and that the eigenvalues of $M$ are both primitive $k\nth$
roots of unity for some $k > \ell$. Then $M$ is one of the following, where $a,b$ are $\ell\nth$ roots of unity.
\begin{enumerate}
\item $M_{1} = \left( \begin{array}{ll} a & ab \\ -ab^{\ast} & a \end{array}\right)$, and $\ell = 2^{\alpha}t$ where $t$ is odd. Then $k > \ell$ when $\alpha \leq 2$, and both eigenvalues have the same order when $\alpha \neq 3$.

\item $M_{2} = \left( \begin{array}{ll} a & ab \\ -iab^{\ast} & ia \end{array}\right)$, and $\ell = 2^{\alpha}3^{\beta}t$ where $t$ is coprime to $6$. Then $k > \ell$ when $\beta \leq 1$ or $\alpha \leq 3$, and both eigenvalues have the same order when $\beta \neq 1$ and $\alpha \neq 3$.
\end{enumerate}
\end{theorem}

\begin{proof}
Throughout we write $\lambda_{1},\lambda_{2}$ for the eigenvalues of $M$, and $m_{ij}$ for the entry in row $i$ and column $j$.
By Corollary \ref{cor:eigratio} and Lemma \ref{traceless}, we may assume that (up to relabelling of eigenvalues and negation of $M$) that
$\lambda_{1}\lambda_{2}^{-1} \in \{i, \zeta_{3}\}$.

\textbf{1:} Suppose $\lambda_{1} = i\lambda_{2}$. Then writing $\lambda_{1} = \sqrt{2} \omega$ for some $\omega \in \mathbb{C}$ of modulus $1$,
we have that $m_{11} + m_{22} = 2\zeta_{8}\omega$, where $\zeta_{8}\omega$.
By the triangle inequality, $a := m_{11} = m_{22}$.  Enforcing the orthogonality of the rows of $M$,
we find that $m_{12} = ab$ and $m_{21} = -ab^{\ast}$ for arbitrary $b$ of modulus $1$.  Thus under the assumption that $M \in \But(2, \ell)$,
we have that $M = M_{1}$. It is easily verified that different choices of $b$ produce similar matrices,
and the eigenvalues of $M_{1}$ are $\sqrt{2}\zeta_{8}a$ and $\sqrt{2}\zeta_{8}^{7} a$.

For convenience, we write $M_{1}(a)$ for the unitary matrix with $2^{-1/2}a$ on the diagonal,
and up to similarity we can take $b = 1$. The group generated by $M_{1}(a)$ is finite and cyclic,
and so a direct product of a cyclic $2$-group and group of odd order.
Since $M_{1}(a)^{4} = -a^{4}I_{2}$ a maximal subgroup of odd order is scalar.
Squaring is an automorphism on the roots of unity of odd order, hence there
exists some root of unity of odd order $\zeta \in \langle \zeta_{\ell}\rangle$
such that $M_{1}(a) = \zeta M_{1}(a')$, where $a'$ is a root of unity of order $2^{\alpha}$.

It is easily verified that when $a' \in \{\pm 1, \pm i\}$ that the eigenvalues of $M_{1}(a')$ are
primitive $8\nth$ roots of unity. But when $a'$ is a primitive $8\nth$ root, then one eigenvalue
of $M_{1}(a')$ is real while the other is purely imaginary. Finally, $a'$ is a primitive root of unity
of order $2^{\alpha}$ for $\alpha \geq 3$, then both eigenvalues of $M_{1}(a')$ are again of order $2^{\alpha}$.

We write $\mathbb{Q}[\zeta_{\ell}]$ for the coefficient field of $\sqrt{2}M_{1}(a)$, where $\ell = 2^{\alpha}t$ and $t$ is odd.
Orthogonality of the rows of $M_{1}(a)$ implies that $\alpha \geq 1$. When $\alpha \in \{1, 2\}$, the eigenvalues of $M_{1}(a)$
are both primitive roots of unity of order $8t$, and so $k > \ell$. When $\alpha = 3$
the eigenvalues have different orders, and in all cases, $M_{1}(a)^{4t} = I_{2}$, while when $\alpha \geq 4$ we always have $k \leq \ell$.

\textbf{2:} Suppose that $\lambda_{1} = \zeta_{3}\lambda_{2}$.
Then the sum of the eigenvalues has modulus $\sqrt{2}$, from which we conclude that $a := m_{11} = im_{22}$.
As before, we can solve for the off diagonal entries in terms of a single unknown, obtaining $m_{12} = ab$ and $m_{21} = -iab^{\ast}$.
Thus under the assumption that $M \in \But(2, \ell)$, we have that $M = M_{2}$. The eigenvalues of $M_{2}$ are $\lambda_{1} = \sqrt{2}\zeta_{24}^{7}a$
and $\lambda_{2} = \sqrt{2}\zeta_{24}^{23}a$.

For convenience, we write $M_{2}(a)$ for the unitary matrix with $2^{-1/2}a$ and $2^{-1/2}ia$ on the diagonal, again up to similarity we can take $b = 1$.  The group generated by $M_{2}(a)$ is finite and cyclic when $a$ is a root of unity, and contains a scalar subgroup of index $3$ generated by $(2^{-3/2}a^{3}(2-2i))I_{2}$.

Suppose now that $a$ is a primitive $\ell\nth$ root of unity where $\ell = 2^{\alpha}3^{\beta}t$, with $\gcd(t, 6) = 1$.
For each choice of $\alpha, \beta, t$, it is routine to compute the orders of the eigenvalues, though there are a large number of
cases to consider. Since $-ia$ is an entry of $\sqrt{2}M_{2}(a)$, we can assume that $\alpha \geq 2$. Next we will show that the eigenvalues of
$M_{2}(a)$ have distinct orders if and only if $\beta = 1$ or $\alpha = 3$.

Suppose that $\beta = 1$: then the eigenvalues of $M_{2}(a)$ are $\zeta_{24\ell}^{24i + 7\ell}$ and $\zeta_{24\ell}^{24i + 23\ell}$ where $i \equiv 1, 2 \mod 3$ and $\ell \equiv 3, 6 \mod 9$ (since $a$ is a primitive $\ell\nth$ root). Suppose that $i \equiv 1 \mod 3$ and $\ell \equiv 3 \mod 9$, then $24i + 7 \ell \equiv 0 \mod 9$ while $24 i + 23 \ell \equiv 3 \mod 9$. But since $9 \mid 24 \ell$, the eigenvalues do not have the same order. The remaining three cases are similar; in no case do the eigenvalues of $M$ have the same order. When $a$ is a primitive $8\nth$ root, $M_{2}(a)$ has order $3, 6$ or $12$ (depending on the choice of primitive root). When $\ell = 2^{3}t$ with $\gcd(t, 6) = 1$ we obtain matrices of orders $3t, 6t, 12t$. And by a similar argument, when $\ell = 2^{3} 3^{\beta} t$ with $\beta \geq 2$ we obtain matrices of order $3^{\beta}t, 2 \cdot 3^{\beta}t,4 \cdot 3^{\beta}t$.

It remains to examine the cases $\alpha = 2$ and $\beta \neq 1$, and $\alpha \geq 4$ with $\beta \neq 1$. In the first case, when $\alpha = 2$ and $\beta = 0$, we have that $a = \zeta_{24}^{j}\zeta_{t}$ for $j \in \{6, 18\}$. So the eigenvalues of $M(a)$ are $\zeta_{24}^{j+7}\zeta_{t}$ and $\zeta_{24}^{j -1}\zeta_{t}$. For each choice of $j$, one finds that both exponents are coprime to 24, and hence the eigenvalues have multiplicative order $24t = 6\ell$. The remaining computations are all similar, in each case one finds that both eigenvalues are primitive roots of order $\lcm(24, \ell)$.
\end{proof}

Suppose that $M = M_{1}$ where $a, b$ are chosen to be $2t\nth$ roots of unity for odd $t$. Then the eigenvalues of $M_{1}$ are $Y = \{\zeta_{8}a, \zeta_{8}^{7}a\}$
and $\sqrt{2}^{1-i}M^{i} \in \But(2, 2t)$ whenever $i$ is odd. Hence the pair $(H, M)$ is sound whenever the entries of $H$ lie in $X = \{ \zeta_{8t}^{2i+1} \mid i \in \mathbb{N}\}$.
We have recovered the generalisation of Turyn's morphism described in Theorem \ref{thm:newfromold}. When $\alpha \geq 3$, then $k \leq \ell$, and we never obtain an interesting morphism. When $\alpha = 2$, we obtain a partial morphism with the same domain as when $\alpha = 1$; but the image lies over a larger field. These are the only non-trivial morphisms obtained from $M_{1}$.

We can also examine the morphisms obtained from $M_{2}$: an interesting case occurs when $a = b = 1$, denote this matrix by $M_{24}$. Its entries are fourth roots,
and its eigenvalues are primitive $24\nth$ roots. It is easily verified that $X = \{\zeta_{24}^{3i+j} \mid i \in \mathbb{N}, j = 1, 2\}$ and
$Y = \{ \zeta_{24}^{7}, \zeta_{24}^{23} \}$. The set $X$ consists of all $24\nth$ roots which are not $8\nth$ roots.
We obtain the obvious partial morphism $\But(n, 24)\rightarrow \But(n, 4)$. Since $X$ contains an arithmetic progression of length $8$,
we have that the pair $(\zeta_{24}H, M_{24})$ is sound for any $H \in \But(n, 8)$. Hence we obtain a complete morphism $\But(n, 8) \rightarrow \But(2n, 4)$
which does not seem to have appeared previously in the literature.

\subsection*{Acknowledgements}

The first author has been fully supported by Croatian Science Foundation under the project 1637.
The argument for Proposition \ref{prop:roots} was provided by Richard Stanley in response to a question posted on MathOverflow.

\bibliographystyle{abbrv}
\flushleft{
\bibliography{NewBiblio}
}
\end{document}